\documentclass{amsart}

\usepackage{epsfig,color}
\usepackage[latin1]{inputenc}

\newtheorem{theorem}{Theorem}[section]
\newtheorem{lemma}[theorem]{Lemma}

\newtheorem{proposition}[theorem]{Proposition}
\newtheorem{definition}[theorem]{Definition}

\newtheorem{remark}[theorem]{Remark}

\numberwithin{equation}{section}

\DeclareMathOperator{\TR}{TR}
\DeclareMathOperator{\tr}{tr}

\newcommand{\cz}{{\mathbb C}}

\newcommand{\gz}{{\mathbb Z}}

\newcommand{\nz}{{\mathbb N}}
\newcommand{\rz}{{\mathbb R}}

\newcommand{\calC}{\mathcal{C}}
\newcommand{\calD}{\mathcal{D}}

\newcommand{\calL}{\mathcal{L}}

\newcommand{\calP}{\mathcal{P}}

\newcommand{\calS}{\mathcal{S}}

\newcommand{\forget}[1]{}

\newcommand{\ci}{\mathcal{C}^\infty}
\newcommand{\cicomp}{\mathcal{C}^\infty_{\text{\rm comp}}}
\newcommand{\cl}{\mathrm{cl}}
\newcommand{\dbar}{d\hspace*{-0.08em}\bar{}\hspace*{0.1em}}
 
\newcommand{\eps}{\varepsilon}
\newcommand{\entier}[1]{\mbox{$[\hspace*{-0.4ex}[ #1 ]\hspace*{-0.4ex}]$}}
\newcommand{\fpint}{{\int\hspace*{-10pt}-\hspace*{5pt}}}
\newcommand{\fpiint}{\fpint\!\!\!\!\fpint}
\newcommand{\im}{\text{\rm Im}\,}

\newcommand{\re}{\text{\rm Re}\,}

\newcommand{\spk}[1]{[#1]}

\newcommand{\st}{\mbox{\boldmath$\;|\;$\unboldmath}}
\newcommand{\smmu}{S^{\mu,m}}

\newcommand{\wt}{\widetilde}

\begin{document}

\title{Determinants of Classical SG-Pseudodifferential Operators}

\author{L.\ Maniccia}

\author{E.\ Schrohe}
\address{Leibniz Universit\"at Hannover, Institut f\"ur Analysis, Welfengarten 1, 
30167 Hannover, Germany}
\email{schrohe@math.uni-hannover.de}

\author{J.\ Seiler}
\address{Universit\'a di Torino, Dipartimento di Matematica, V. Carlo Alberto 10, 10123 Torino, Italy}
\email{joerg.seiler@unito.it}



\begin{abstract}
We introduce a generalized trace functional $\TR$  in the spirit of Kontsevich and Vishik's canonical 
trace for classical $SG$-pseudodifferential operators 
on $\rz^n$ and suitable manifolds, using a finite-part integral 
regularization technique. This allows us to define a zeta-regularized determinant $\det A$ for 
parameter-elliptic operators $A\in S^{\mu,m}_\cl$, $\mu>0$, $m\ge0$. For $m=0$, the asymptotics 
of $\TR e^{-tA}$ as $t\to 0$ and of $\TR (\lambda-A)^{-k}$ as $|\lambda|\to\infty$ are derived. 
For suitable pairs $(A,A_0)$ we show that $\det A/\det A_0$ coincides with the so-called relative 
determinant $\det(A,A_0)$. 
\end{abstract}

\maketitle

\section{Introduction}\label{sec:zero}
In 1971, Ray and Singer \cite{RaSi} introduced a new notion of regularized 
determinants. It was based on Seeley's analysis of complex powers of elliptic 
operators \cite{Seel}: Assuming that $A$ is an invertible classical
pseudodifferential operator of positive order on a closed manifold 
and that there exists a ray 
$\{r\exp(i\phi):r>0\}$ in the complex plane (for some $0\le\phi<2\pi$), on 
which the principal symbol of $A$ has no eigenvalue, Seeley showed that the 
complex powers $A^z$, $z\in\cz$, can be defined. Moreover, he proved that the 
associated zeta function $\zeta_A$ given by $\zeta_A(z)=\tr(A^{-z})$ for  $\re (z)$
sufficiently large, extends to a meromorphic function on $\cz$ which is regular in $z=0$.
It therefore makes sense to let, as Ray and Singer suggested,   
$$\det A = \exp\Big(-\frac{d}{dz}\Big|_{z=0}\zeta_A(z)\Big)$$
thus extending the corresponding identity for finite matrices. 

Since then, regularized determinants have found applications in various areas of mathematics 
and physics. The articles \cite{BGKE}, \cite{HaZe}, \cite{Hawk}, \cite{KoVi2}, \cite{LoPa1}, 
\cite{Mull1}, \cite{Okik}, \cite{OPS}, \cite{RaSi}, \cite{Sher} give an impression of the diversity of 
applications. For an even more detailed account on the literature we refer the reader to 
M\"uller \cite{Mull}. 

This approach fails, if the underlying manifold is non-compact.
In this paper, we therefore focus on operators on $\rz^n$ with additional structure, 
the  classical $SG$-pseudodifferential operators (sometimes also called 
scattering operators). 
A symbol $a$ here is required to  satisfy the
estimates 
\begin{equation}\label{zeroi.B}
 |D^\alpha_\xi D^\beta_x a(x,\xi)|\le C_{\alpha,\beta}
 (1+|x|)^{m-|\beta|}(1+|\xi|)^{\mu-|\alpha|},\quad x,\xi\in\rz^n, 
\end{equation}
for all $\alpha,\beta\in\nz_0^n$ with suitable $\mu$ and $m$. We then write 
$a\in S^{\mu,m}$. Classical symbols additionally 
have asymptotic expansions with respect to both $x$ and $\xi$. 
For a short introduction  see the Appendix.
A typical example for such an operator is a differential operator 
$A(x,D)=\sum_{|\alpha|\le \mu}c_\alpha D^\alpha$ of order $\mu$, whose
coefficient functions $c_\alpha=c_\alpha(x)$ are classical symbols in $x$ of order $m$ 
on $\rz^n$.   

We first define the generalized trace functional $\TR$ via a finite part integral 
and study its basic properties.
We then turn to the case where $\mu>0$ and $m=0$ in \eqref{zeroi.B}. 
Assuming that the symbol $a(x,\xi)$ is parameter-elliptic with respect to the sector 
\begin{equation}\label{zeroi.C}
 \Lambda=\Lambda(\theta)=
 \{re^{i\varphi}\st r\ge 0\text{ and }\theta\le\varphi\le2\pi-\theta\}
\end{equation}
(where $0<\theta<\pi$ is fixed) and that $A=a(x,D)$ has no spectrum in $\Lambda$,
 we introduce a regularized zeta function $\zeta_A(z)=\TR(A^{-z})$ of $A$.
Using the description of the complex powers given in \cite{MSS}, we show that $\zeta_A$ 
is meromorphic in the complex plane and holomorphic in 0. 
Thus we can define the determinant of $A$ as  $\det A=\exp(-\zeta^\prime_A(0))$. 

In case $\mu,m>0$ in \eqref{zeroi.B}, suitable ellipticity assumptions will ensure that 
the complex powers $A^{-z}$ exist and are of trace class whenever the real part of $z$ is 
large enough. We can therefore define $\zeta_A(z)=\tr (A^{-z})$ for those $z$. 
Again, this function has a meromorphic extension to the whole complex plane. 
In general, it will have a simple pole at the origin. 
A natural definition for the determinant in this situation is 
\begin{equation}\label{zeroi.D}
 \mathrm{det}\,A=\exp\Big(-\mathrm{res}_{z=0}\frac{\zeta_A(z)}{z^2}\Big).
\end{equation}
Certain symmetry properties of the symbol will imply the regularity of $\zeta_A$
in zero. 
The usual definition of the zeta-regularized  determinant then applies and agrees
with \eqref{zeroi.D}. 

M\"uller \cite{Mull} introduced the concept of  the \textit{relative determinant}
$\det(A,A_0)$ for suitable pairs $(A,A_0)$ of self-adjoint operators. 
He assumes that the trace of the difference of the heat kernels of $A$ and $A_0$ exists 
and has a certain asymptotic behavior both as $t\to 0$ and $t\to\infty$. 
He then defines the \textit{relative $\zeta$-function} 
$$\zeta(z;A,A_0)=\frac1{\Gamma(z)}\int_0^\infty t^{z-1}\tr(e^{-tA}-e^{-tA_0})dt$$ 
and lets
 $$\det(A,A_0)=\exp\Big(-\frac{d}{dz}\Big|_{z=0}\zeta(z;A,A_0)\Big).$$
In case $A$ and $A_0$ are positive self-adjoint pseudodifferential operators on a closed 
manifold, this relative determinant coincides with the ratio of the $\zeta$-regularized 
determinants of $A$ and $A_0$ in the sense of Ray and Singer.

We obtain a similar statement for $SG$-pseudodifferential operators. 
If  $A$ and $A_0$ have symbols in $S^{\mu,0}$, $\mu>0$, and simultaneously satisfy 
our ellipticity assumptions as well as the assumptions for defining their relative determinant, 
then  
$$\mathrm{det}(A,A_0)=\frac{\det A}{\det A_0}.$$ 
To establish this relation we describe the asymptotic structure of 
$\TR(\lambda-A)^{-k}$ for sufficiently large $k\in\nz$ as $|\lambda|\to\infty$,
using results of Grubb and Seeley  \cite{GrSe}, and relate the short time asymptotics of 
$\TR e^{-tA}$ to the pole structure of the generalized zeta function of $A$. 

We finally point out how these results apply to classical SG-pseudodifferential operators
on manifolds with ends. 

It is a challenging question -- beyond the scope of the present article -- to what extent the results obtained here hold for larger classes of pseudodifferential operators on noncompact manifolds. 
Of particular interest are the general pseudodifferential 
calculi devised by Ammann, Lauter and Nistor \cite{ALN} and Monthubert \cite{Mont}.  
The zeta-regularization technique requires a very precise understanding 
of the complex powers of elliptic operators and their symbols, which has not yet been established for these operators. 
Ammann, Nistor, Lauter and Vasy  \cite{ALNV} have shown that, indeed, complex powers can be constructed in a fairly general setting, 
but the article does not contain sufficient information on their structure.  
Moreover, the zeta regularization approach relies on the existence of suitable regularized traces, which is not clear in general -- 
but a worthwhile subject in its own right, see e.g. Paycha \cite{Payc}. 

Another task is the construction of \textit{multiplicative} 
determinants in the spirit of Scott's residue determinant \cite{Scott} for SG-manifolds 
and their classification, in an extension of the work of Lescure and Paycha \cite{LePa}. 
This is left to future investigation.   \bigskip

\section{The Finite-part Integral, Regularized Trace-functionals and Determinants}\label{sec:zero1}
\subsection{The Finite-part Integral}\label{sec:zero1.1}
This concept allows the formal integration of functions
with polyhomogeneous expansions. For simplicity let 
$b=b(y)\in S^\ell_{\text{\rm cl}}(\rz^n)$, $\ell \in \cz$. Its finite part-integral is 
\begin{equation}\label{zero1.A}
 \fpint b(y)\,dy:=\lim_{R\to\infty}
 \Big\{\int_{|y|\le R}b(y)\,dy-\sum_{j=0}^{L}\frac{\beta_j}{n+\ell-j}R^{n+\ell-j}\Big\},
 \qquad \ell\notin\gz_{\ge-n},
\end{equation} 
and 
\begin{equation}\label{zero1.B}
 \fpint b(y)\,dy:=\lim_{R\to\infty}
 \Big\{\int_{|y|\le R}b(y)\,dy-\sum_{j=0}^{L-1}\frac{\beta_j}{n+\ell-j}R^{n+\ell-j}-
 \beta_{n+\ell}\ln R\Big\},
 \qquad \ell\in\gz_{\ge-n},
\end{equation} 
respectively, where $L=\entier{n+\re\ell}$ $($here, and in the following, $\entier{t}$ 
denotes the integer part of $t\in\rz)$ and 
\begin{equation}\label{zero1.C}
 \beta_j=\int_{|y|=1}b_{(\ell-j)}(y)\,d\sigma(y),\qquad j\in\nz_0,  
\end{equation}
where $d\sigma(y)$ is the surface measure on the unit sphere in $\rz^n$ and $b_{(\ell-j)}$ 
is the component of homogeneity $\ell-j$ in the asymptotic expansion of $b$. The finite-part 
integral is a continuous functional on $S^\ell_\cl(\rz^n)$. It coincides with the usual integral 
in case $\re \ell<-n$. 

To verify that the above limits exist choose a zero-excision function $\chi(y)$ which is 
identically 1 outside the unit ball and let $M\ge L$ be an integer. Then, for $R\ge1$, 
\begin{align*}
 \int_{|y|\le R}b(y)\,dy=&
 \int_{|y|\le R}\Big\{b(y)-\chi(y)\sum_{j=0}^{M} b_{(\ell-j)}(y)\Big\}\,dy+
 \sum_{j=0}^{M}\int_{|y|\le R} \chi(y)b_{(\ell-j)}(y)\,dy.
\end{align*} 
Introducing polar coordinates results in  
\begin{align*}
 \int_{|y|\le R} \chi(y)b_{(\ell-j)}(y)\,dy=\int_{|y|\le 1} \chi(y)b_{(\ell-j)}(y)\,dy
 +\beta_j\int_1^R r^{n+\ell-j-1}\,dr
\end{align*} 
from which the claim follows immediately. We also see the following: 

\begin{remark}\rm\label{zero1.1.5}
For $b\in S^\ell_\cl(\rz^n)$ and an arbitrary integer $M\ge\entier{n+\re \ell}$
\begin{align*}
  \fpint b(y)\,dy=
  \int\Big\{b(y)-\chi(y)\sum_{j=0}^{M} b_{(\ell-j)}(y)\Big\}\,dy
  +\sum_{j=0}^{M}
  \Big\{\int_{|y|\le 1} \chi(y)b_{(\ell-j)}(y)\,dy-\frac{\beta_j}{n+\ell-j}\Big\} 
\end{align*}
provided $\ell\not\in\gz_{\ge-n}$. 
Otherwise the term for $j=n+\ell$ in the last sum has to be replaced by 
$\displaystyle\int_{|y|\le1} \chi(y)b_{(-n)}(y)\,dy$. 
\end{remark}

Now let $a\in\smmu_\cl$. Recall that we denote by $a^{(\mu-k)}$, $a_{(m-j)}$, and 
$a^{(\mu-k)}_{(m-j)}$ the $\xi$-, $x$-, and mixed homogeneous 
components of $a$, cf.\ the Appendix. Finite part integration with respect to 
$\xi$ results in a classical symbol 
\begin{equation}\label{zero1.E}
 d(x)=\fpint a(x,\xi)\,\dbar\xi\;\in\; S^m_\cl(\rz^n_x) 
\end{equation} 
(where $\dbar\xi=(2\pi)^{-n}d\xi)$ with the homogeneous components 
\begin{equation*}
 d_{(m-j)}(x)= \fpint a_{(m-j)}(x,\xi)\,\dbar\xi, \qquad j\in\nz_0. 
\end{equation*} 

\begin{proposition}\label{zero1.3}
Let $a\in\smmu_\cl$ with arbitrary $\mu,m\in\cz$. 
Then, for any integers $N\ge\entier{n+\re\mu}$ and $M\ge\entier{n+\re m}$,  
\begin{align*}
 \fpint\!\!\!\!\fpint a\,dx\dbar\xi=&
 \iint\Big\{a-\kappa\sum_{k=0}^{N}a^{(\mu-k)}-\chi\sum_{j=0}^{M}a_{(m-j)}+
   \kappa\chi\sum_{j=0}^{M}\sum_{k=0}^{N}a^{(\mu-k)}_{(m-j)}\Big\}dx\dbar\xi+\\
 &+\sum_{j=0}^{M}
   \iint_{|x|\le1}\chi\Big\{a_{(m-j)}-\kappa\sum_{k=0}^{N}a^{(\mu-k)}_{(m-j)}
   \Big\}dx\dbar\xi+\\
 &+\sum_{k=0}^{N}
   \int_{|\xi|\le1}\int\kappa\Big\{a^{(\mu-k)}-\chi\sum_{j=0}^{M}a^{(\mu-k)}_{(m-j)}
   \Big\}dx\dbar\xi+\\
 &+\sum_{j=0}^{M}\sum_{k=0}^{N}
   \int_{|\xi|\le1}\int_{|x|\le1}\kappa\chi a^{(\mu-k)}_{(m-j)}dx\dbar\xi-\\
 &-\sum_{k=0}^{N}\gamma(n+\mu,k)\fpint\!\!\!\!\int_{|\xi|=1}a^{(\mu-k)}\,
  d\sigma(\xi)\dbar x-\\
 &-\sum_{j=0}^{M}\gamma(n+m,j)\fpint\!\!\!\!\int_{|x|=1}a_{(m-j)}\,
  d\sigma(x)\dbar\xi-\\
 &-\sum_{j=0}^{M}\sum_{k=0}^{N}\frac{\gamma(n+\mu,k)\gamma(n+m,j)}{(2\pi)^{n}}
  \int_{|\xi|=1}\int_{|x|=1}a^{(\mu-k)}_{(m-j)}\,d\sigma(x)d\sigma(\xi), 
\end{align*}
where we have set 
$\gamma(s,t)=\begin{cases}\frac{1}{s-t},&t\not=s\\ 0&\text{\rm else}\end{cases}$
and have  omitted the argument $(x,\xi)$. 
Moreover, $\chi(x)$ and $\kappa(\xi)$ denote zero-excision functions 
which are identically 1 outside the unit ball. 
\end{proposition}

The proof is straightforward, and we omit the details. In particular we see that, for $a\in\smmu_\cl$,
\begin{equation}\label{eq:fubini}
 \fpint\!\!\!\!\fpint a(x,\xi)\,\dbar\xi dx=\fpint\!\!\!\!\fpint a(x,\xi)\,dx\dbar\xi.
\end{equation}

\subsection{Relation to the Kontsevich-Vishik Density}\label{sec:zero1.2}
Let  $a\in\smmu_\cl$ with $\mu\notin\gz_{\ge-n}$. 
Kontsevich and Vishik  \cite{KoVi1}, \cite{KoVi2},
associated to the operator $a(x,D)$ a density, which is a 
regularization of the Schwartz kernel $K=K(x,y)$ on the diagonal $y=x$.
In fact, $K(x,y)=\widehat a(x,y-x)$, where $\widehat a(x,\cdot)$ is the Fourier transform 
of $a(x,\cdot)$ for fixed $x$. Outside the diagonal, $K$ is smooth, 
and the expansion of $a(x,\cdot)$ into terms $a^{(\mu-k)}(x,\cdot)$ of homogeneity 
$\mu-k$ in $\xi$ furnishes an expansion of $K$ into terms $K_{k}(x,\cdot)$, which are 
homogeneous of degree $-n-\mu+k$ in $y-x$; here it is important that $\mu\notin \gz_{\ge-n}$,
see Theorems 3.2.3 and 7.1.18 in \cite{Horm} for details.
The regularized value for $K(x,x)=\widehat a(x,0)$ defined by Kontsevich and Vishik is
\begin{eqnarray}\label{zero1.F}
\mathbf{d}(x)=\int\Big\{a(x,\xi)-\sum_{k=0}^{N}a^{(\mu-k)}(x,\xi)\Big\} \,\dbar\xi.
\end{eqnarray} 
Here,  $N=\entier{n+\re\mu}$, and the expression is to be interpreted as follows: With a $0$-excision 
function $\kappa$ write  
\begin{eqnarray}\label{zero1.J}
\mathbf d(x)= \int\Big\{a(x,\xi)-\kappa(\xi)\sum_{k=0}^{N}a^{(\mu-k)}(x,\xi)\Big\} \,\dbar\xi 
+\int\Big\{(\kappa(\xi)-1)\sum_{k=0}^{N}a^{(\mu-k)}(x,\xi)\Big\} \,\dbar\xi .
\end{eqnarray} 
The first term is a Lebesgue integral, while the second can be considered as 
\begin{eqnarray}\label{zero1.K}
(2\pi)^{-n}\sum_{k=0}^{N}\left\langle a^{(\mu-k)}(x,\cdot), (\kappa -1)
\right\rangle
\end{eqnarray}
with the duality between $\mathcal S'$ and $\mathcal S$.
For $k>\re \mu+n$,  $K_k(x,x)=\widehat a^{(\mu-k)}(x,0)$ vanishes; hence $N$ can be increased without changing the result.

The function $\mathbf{d}$ in fact is a classical symbol of order $m$ whose homogeneous 
components are 
\begin{equation*}
 \mathbf{d}_{(m-j)}(x)= \int\Big\{a_{(m-j)}(x,\xi)-\sum_{k=0}^{N-1}
 a^{(\mu-k)}_{(m-j)}(x,\xi)\Big\}\,\dbar\xi,\qquad j\in\nz_0,
\end{equation*} 
where the meaning of the integration is as in \eqref{zero1.F}. 

For a pseudodifferential operator of order $\mu\notin \gz_{\ge-n}$ 
on a manifold, the local expressions $\mathbf{d}(x)dx$ yield a globally defined 
density $\omega_{KV}$, the Kontsevich-Vishik density on the manifold, 
see Lesch \cite[Section 5]{Lesc} for more details.

\begin{proposition}\label{zero1.5}
Let $a\in\smmu_\cl$ with $\mu\notin\gz_{\ge-n}$ and $\mathbf{d}$ given by \eqref{zero1.F}. Then
\begin{equation*}
 \mathbf{d}(x)=\fpint a(x,\xi)\,\dbar\xi. 
\end{equation*}
\end{proposition}
\begin{proof}
Recall how the function $a^{(\mu-k)}(x,\cdot)$ is extended to a homogeneous distribution 
of degree $\mu-k$ on $\rz^n$, cf.\ (3.2.23) in \cite{Horm}: For
an arbitrary $\psi\in\calC^\infty_{\rm comp}(\rz_+)$ satisfying 
$\displaystyle\int_0^\infty\psi(s)\frac{ds}{s}=1$ let
 $$\big\langle a^{(\mu-k)}(x,\cdot),\kappa-1\big\rangle=
     \big\langle a^{(\mu-k)}(x,\cdot),\psi(|\cdot|)R_{\mu-k}(\kappa-1)\big\rangle$$
with 
 $$R_{\mu-k}(\kappa-1)(\xi)=-\frac{|\xi|}{n+\mu-k}\int_0^\infty t^{n+\mu-k}
   \kappa^\prime(t|\xi|)\,dt.$$
In polar coordinates we obtain 
 $$\big\langle a^{(\mu-k)}(x,\cdot),\kappa-1\big\rangle=-\frac{1}{n+\mu-k}\int_{|\xi|=1}
   a^{(\mu-k)}(x,\xi)\,d\sigma(\xi)\;\int_0^\infty r^{n+\mu-k}\kappa^\prime(r)\,dr.$$
As $\kappa^\prime(r)=0$ for $r\ge1$, integration by parts yields that 
$\big\langle a^{(\mu-k)}(x,\cdot),\kappa-1\big\rangle$ equals 
\begin{align*}
 &\frac{1}{n+\mu-k}\int_{|\xi|=1}
   a^{(\mu-k)}(x,\xi)\,d\sigma(\xi)\Big\{-1+
   (n+\mu-k)\int_0^1 r^{n+\mu-k-1}\kappa(r)\,dr\Big\}\\
 =&-\frac{1}{n+\mu-k}\int_{|\xi|=1}a^{(\mu-k)}(x,\xi)\,d\sigma(\xi)+
  \int_{|\xi|\le1}\kappa(\xi)a^{(\mu-k)}(x,\xi)\,d\xi.
\end{align*}
Thus \eqref{zero1.K} coincides with  
 $$-\frac{1}{(2\pi)^n}\sum_{k=0}^{N-1}\frac{1}{n+\mu-k}
   \int_{|\xi|=1}a^{(\mu-k)}(x,\xi)\,d\sigma(\xi)+
   \sum_{k=0}^{N-1}\int_{|\xi|\le1}\kappa(\xi)a^{(\mu-k)}(x,\xi)\,\dbar\xi$$
and the claim follows by comparison with the formula given in Remark \ref{zero1.1.5}. 
\end{proof}

\subsection{Regularized Zeta Function and Determinant}\label{sec:zero1.3}

We shall proceed to define a continuous trace functional $\TR$ acting on different symbol classes.

\begin{definition}\rm\label{zero1.6}\rm 
For $\mu\in\cz\setminus\gz_{\ge-n}$ let
 $$\TR:S^{\mu,0}_\cl\longrightarrow\cz,\quad a\mapsto\fpiint a(x,\xi)\,dx\dbar\xi.$$
In case $a$ takes values in $k\times k$ matrices, we define $\TR a$ as the regularized integral 
over the matrix trace of $a$ without indicating this in the notation. 
\end{definition}

\begin{remark}\rm\label{rem:kernel}
\begin{itemize}
 \item[$($a$)$] 
   For $\re\mu<-n$ this simplifies to 
   \begin{eqnarray}\label{zero1.6A}
      \TR   a=\int\!\!\!\fpint a(x,\xi)\,dx\dbar\xi,
   \end{eqnarray}
   i.e., regularization is needed only in the $x$-variable. Interchanging order of integration, 
   cf.\ \eqref{eq:fubini}, we can also write 
    $$\TR   a=\fpint K_a(x,x)\,dx,$$
   with the kernel $K_a(x,y)$ of $a(x,D)$. 
 \item[$($b$)$] For  $\re\mu<-n$, \eqref{zero1.6A} extends to the slightly larger 
   class $S^{\mu,0}_{\cl(x)}:=S^0_\cl\big(\rz^n_x,S^\mu(\rz^n_\xi)\big)$ 
   $($cf. Section \ref{sec:zero5.0} for the definition of this space$)$. 
\end{itemize}
\end{remark}

We shall refer to $\TR$ as a
\textit{regularized trace}. This is justified, since it 
coincides with the usual trace of the trace-class operator $a(x,D)\in\calL(L^2(\rz^n))$ 
provided the orders of the symbol both in $x$ and $\xi$ are strictly less than $-n$. 

\begin{definition}\rm
Given holomorphic functions $m=m(z)$ and $\mu=\mu(z)$ on an open 
subset $U$ of $\cz$, we say that $\{a(z)=a(x,\xi;z)\st z\in U\}$ is a holomorphic
family of symbols in $S^{\mu(z),m(z)}_{\rm cl}$ on $U$, provided that 
$$z\mapsto [\xi]^{-\mu(z)}[x]^{-m(z)}a(x,\xi;z)$$
is a holomorphic map from $U$ to $S^{0,0}_{\rm cl}$. See the beginning of Section  $\ref{sec:zero5}$
for the definition of $[\cdot]$.
\end{definition}

The operators considered in \cite{ALNV} include the SG-pseudodifferential operators 
with symbols in $S^{\mu,0}$. The `special holomorphic families' introduced there  
correspond to holomorphic families in the above sense with the particular choice 
$m(z)\equiv 0$ and $\mu(z)=mz+d$.

Applying the regularized trace to holomorphic families of pseudodifferential operators 
yields meromorphic functions on the complex plane. The pole structure can be described 
with the help of  the explicit formulas for the finite-part integral in Proposition 
\ref{zero1.3}. 

\begin{theorem}\label{zero1.7}
Let $a(x,\xi;z)\in S^{\mu(z),0}_\cl$ be a holomorphic family on an open subset 
$U$ of $\cz$ with $\mu\not\equiv{\rm const}$ and 
 $$\calP=\{z\in U \st \mu(z)+n\in\nz_0\}.$$
Then $z\mapsto\TR a(z)$ is a meromorphic function on $U$ with poles at most 
in $\calP$. If $z_0\in\calP$ then  
\begin{align}\label{zero1.I}
\begin{split}
\TR a(z)\equiv&\,   
 \frac{(2\pi)^{-n}}{\mu(z_0)-\mu(z)}\fpint\!\!\!\!\int_{|\xi|=1} 
   a^{(\mu(z)-\mu(z_0)-n)}(x,\xi;z)\,d\sigma(\xi)dx+\\
&+\frac{(2\pi)^{-n}}{\mu(z_0)-\mu(z)}\sum_{j=0}^{n-1}\frac{1}{n-j}
   \int_{|\xi|=1}\int_{|x|=1}
   a^{(\mu(z)-\mu(z_0)-n)}_{(-j)}(x,\xi;z)\,d\sigma(x)d\sigma(\xi)
\end{split}
\end{align}
 near $z_0$, modulo a function which is holomorphic in $z_0$.
\end{theorem}

In particular, $z\mapsto \TR a(z)$ is meromorphic in $U$, and the order of a pole 
$z_0\in\calP$ is at most equal to the multiplicity of the zero of 
$\mu(z)-\mu(z_0)$ in $z_0$. We obtain an important special case 
by choosing the holomorphic family as the symbols of the complex powers of 
an elliptic operator; this leads to the so-called regularized zeta function: 

\begin{definition}\rm\label{zero1.8}
Let $a\in S^{\mu,0}_\cl$, $\mu>0$, be $\Lambda$-elliptic, cf.\ Section {\rm\ref{sec:zero5.2}}, 
and assume that the resolvent of $A:=a(x,D)$ exists in the whole sector $\Lambda$. 
Denote by $a(z)$ the symbol of $A^{-z}$.
We call $\zeta_A(z):=\TR a(z)$ the zeta function of $A$. 
\end{definition}

By Theorem \ref{zero1.7},  $\zeta_A$ is meromorphic on 
$\cz$ with simple poles at most in the points $z=\frac{n-k}{\mu}$, $k\in\nz_0$. 
Moreover, $z=0$ is a removable singularity. In fact, since $A^0=1$, the integral 
expressions in \eqref{zero1.I} are holomorphic functions that vanish in zero, while 
the factor $\frac{1}{\mu z}$ has a simple pole in zero. In particular, we can evaluate 
the derivative of the zeta function in $z=0$, and introduce the determinant 
of $A$:

\begin{definition}\rm\label{zero1.9}
Let $a\in S^{\mu,0}_\cl$, $\mu>0$, be $\Lambda$-elliptic, and 
assume that the resolvent of $A:=a(x,D)$ exists in the whole sector $\Lambda$. Then define 
 $${\rm det}\,A:=\exp\Big(-\frac{d}{dz}\Big|_{z=0}\zeta_A(z)\Big).$$
\end{definition}

A simple example is that of a $\Lambda$-elliptic symbol $a=a(\xi)\in S^{\mu,0}_\cl$  independent 
of $x$. In this case  $$\mathrm{det}\,a(D)=1,$$
since the symbol of the $A^z$ is also independent of $x$ and hence its regularized trace vanishes 
$($since the finite part-integral of a constant equals zero$)$.

\section{The Resolvent and the Asymptotics of its Regularized Trace}\label{sec:zero2}

The purpose of this section is to study the asymptotic behavior of 
$\TR(\lambda-a(x,D))^{-k}$  as $|\lambda|\to\infty$ for a $\Lambda$-elliptic symbol 
$a(x,\xi)\in S^{\mu,0}_\cl$, where $\mu\in\nz$ and $k\mu>n$. The following 
theorem holds: 

\begin{theorem}\label{zero2.1}
Let $a(x,\xi)\in S^{\mu,0}_\cl$ with $\mu\in\nz$ be a $\Lambda$-elliptic symbol. 
Let $k\in\nz$ such that $k\mu>n$. Then $(\lambda-a(x,D))^{-k}$ exists for each  
$\lambda\in\Lambda$ of sufficiently large absolute value and is a pseudodifferential 
operator with symbol in $S^{-k\mu,0}_{\cl(x)}$. Moreover, 
for suitable constants $c_j,c_j^\prime,c_j^{\prime\prime}\in\cz$,  
\begin{equation}\label{zero2.A}
\TR(\lambda-a(x,D))^{-k}\sim 
 \sum_{j=0}^\infty c_j\lambda^{\frac{n-j}{\mu}-k}
 +\sum_{j=0}^\infty \big(c_j^\prime\log\lambda+c_j^{\prime\prime}\big)
 \lambda^{-j-k}
\end{equation}
uniformly as $|\lambda|\to\infty$ in $\Lambda$. 
\end{theorem}

The symbol $\sim$ in \eqref{zero2.A} means that, for each $N\in\nz_0$, 
 $$\TR(\lambda-a(x,D))^{-k}-  
 \sum_{j=0}^{N-1} c_j\lambda^{\frac{n-j}{\mu}-k}
 -\sum_{j=0}^{N-1} \big(c_j^\prime\log\lambda+c_j^{\prime\prime}\big)
 \lambda^{-j-k}=O(|\lambda|^{-N/\mu}).$$

For the proof of Theorem \ref{zero2.1} we shall write 
 $$\TR(\lambda-a(x,D))^{-k}=\fpint K(x,x,\lambda)\,dx$$ 
with the kernel $K(x,y,\lambda)$ of $(\lambda-a(x,D))^{-k}$, 
cf.\ Remark \ref{rem:kernel}, 
and establish the following asymptotic structure of the kernel: 

\begin{theorem}\label{zero2.2}
Let $a(x,\xi)\in S^{\mu,0}_\cl$ with $\mu\in\nz$ be a $\Lambda$-elliptic symbol. 
Let $k\in\nz$ such that $k\mu>n$ and $K(x,y,\lambda)$ be the kernel of  
$(\lambda-a(x,D))^{-k}$. 
Then there exist symbols $c_j,c_j^\prime,c_j^{\prime\prime}\in S^0_\cl(\rz^n)$, 
$j\in\nz_0$, such that, for each $N\in\nz_0$, 
 $$|\lambda|^{N/\mu}
   \Big\{
    K(x,x,\lambda)- \sum_{j=0}^{N-1} c_j(x)\lambda^{\frac{n-j}{\mu}-k}
    -\sum_{j=0}^{N-1} \big(c_j^\prime(x)\log\lambda+c_j^{\prime\prime}(x)\big)
    \lambda^{-j-k}
   \Big\}
   \in S^0_\cl(\rz^n_x)$$
is uniformly bounded as $|\lambda|\to\infty$ in $\Lambda$. 
\end{theorem}

This asymptotic structure can be obtained by adapting the approach of 
Grubb and Seeley \cite{GrSe} to the case of classical SG-symbols. 
For the convenience of the reader we shall outline the proof 
in the remaining part of this section. 

We first introduce a class of weakly polyhomogeneous 
symbols that contains the parametrix of $(\sigma^\mu-a(x,D))^{-k}$ for 
$\sigma^\mu\in\Lambda$. 

\begin{definition}\rm\label{zero2.3}
Let $E$ be a Fr\'echet space, $\Sigma$ an open subsector of $\cz\setminus\{0\}$ 
and $\nu\in\rz$. 
A function $p(\xi,\sigma)$ belongs to the symbol class $S^{\nu;0}(\Sigma,E)$, if 
\begin{itemize}
 \item[a)] $p:\rz^n\times\Sigma\to E$ is smooth and holomorphic in $\sigma\in\Sigma$ on
   $\{(\xi,\sigma)\st |\xi|+|\sigma|>\eps\}$ for some $\eps>0$,
 \item[b)] for each $j\in\nz_0$ 
    $$\partial^j_z p\big(\xi,\mbox{$\frac{1}{z}$}\big)\;\in\;
      S^{\nu+j}(\rz^n_\xi,E),$$
  uniformly for $\frac{1}{z}$ in closed subsectors of $\Sigma$ such that $|z|\le 1$. 
  See Section \ref{sec:zero5.0} for the definition of $E$-valued symbols. 
\end{itemize}
For $d\in\rz$, we let $S^{\nu;d}(\Sigma,E)=\sigma^d S^{\nu;0}(\Sigma,E)$. 
For $E=S^{s}_\cl(\rz^n_x)$ we shall write
 $$S^{(\nu,s);d}_\cl=S^{\nu;d}(\Sigma,S^{s}_\cl(\rz^n_x)),\qquad \nu,s,d\in\rz.$$
\end{definition}

For $E=C^\infty(\rz^n_x)$ we recover the symbol spaces used in \cite{GrSe}. 
Note that for $S^{\nu;d}(\Sigma,E)$ the limits 
\begin{equation}\label{zero2.A.5}
 p_{(d,j)}(\xi)=\lim_{z\to 0}\partial_z^j 
 \Big(z^dp\big(\xi,\mbox{$\frac{1}{z}$}\big)\Big)
\end{equation}
exist in $E$ and define symbols $p_{(d,j)}\in S^{\nu+j}(\rz^n,E)$, 
cf.\ Theorem 1.12 in \cite{GrSe}.

\begin{definition}\rm\label{zero2.4}
$p\in S^{\nu_0-d;d}(\Sigma,E)$ is called weakly polyhomogeneous if there exists a 
sequence of symbols $p_j\in S^{\nu_j-d;d}(\Sigma,E)$ that are positively homogeneous 
in $(\xi,\sigma)$ for $|\xi|\ge1$ of degree $\nu_j\searrow-\infty$, such that  
$p-\sum\limits_{j=0}^{N-1}p_j\;\in\;S^{\nu_N-d;d}(\Sigma,E)$ 
for each $N\in\nz_0$. 
We write $p\sim\sum_j p_j$ in $S^{\infty;d}(\Sigma,E)$. 
\end{definition}

Note also that given any sequence $p_j\in S^{\nu_j-d;d}(\Sigma,E)$ with 
$\nu_j\to-\infty$, there always exists a symbol $p\in S^{\nu_0-d;d}(\Sigma,E)$ 
such that $p\sim\sum_j p_j$ in $S^{\infty;d}(\Sigma,E)$. 

Following the proof of Theorem 2.1 in \cite{GrSe}, one can show that  
if $p\in S^{(\nu_0-d,s);d}_\cl$ is a weakly 
polyhomogeneous symbol as in Definition \ref{zero2.4} with the additional property that 
there exist $\nu^\prime,d^\prime\in\rz$ with $\nu^\prime<-n$ and  
$p_j\in S^{(\nu^\prime,s);d^\prime}_\cl$ whenever $\nu_j-d\ge-n$, then the kernel 
$K(x,y,\sigma)$ of $p(x,D,\sigma)$ satisfies the following: There exist symbols 
$\alpha_k,\alpha_k^\prime,\alpha_k^{\prime\prime}\in S^s_\cl(\rz^n)$ such that, 
for each $N\in\nz_0$,
\begin{equation}\label{zero2.B} 
 \sigma^{N-d}
 \Big\{
  K(x,x,\sigma)-\sum_{j=0}^{N-1}\alpha_j(x)\sigma^{\nu_j+n}
  -\sum_{k=0}^{N-1}\big(\alpha_k^\prime(x)\log\sigma+\alpha_k^{\prime\prime}(x)\big)
  \sigma^{d-k}
 \Big\}\;\in\;S^{(0,s);0}_\cl,
\end{equation}
depending only on the variables $(x,\sigma)$. 
If $p\in S^{(-\infty,s);d}$ then in the analogous sense    
\begin{equation}\label{zero2.B.5} 
 \sigma^{N-d}
 \Big\{K(x,x,\sigma)-\sum_{k=0}^{N-1}\alpha_k^{\prime\prime}(x)\sigma^{d-k}\Big\}
 \;\in\;S^{(0,s);0}_\cl. 
\end{equation}
The $\alpha_k^\prime$ and $\alpha_k^{\prime\prime}$ in \eqref{zero2.B} 
are determined by the limit symbols of $p_{j}$ and $p-\sum_{j<N}$, $N\in\nz_0$, 
in the sense of \eqref{zero2.A.5}, respectively; the $\alpha_k^{\prime\prime}$ in \eqref{zero2.B.5} only depend on the 
limit symbols of $p$. 
For details see Proposition 1.21 and Theorem 2.1 in \cite{GrSe}.  

By Remark \ref{rem:ell-classical} there exists a sector $\Lambda^\prime$ whose 
interior contains $\Lambda\setminus\{0\}$ such that $a$ is $\Lambda^\prime$-elliptic. For 
notational convenience we shall denote by $a^{(\mu-j)}(x,\xi)\in S^{\mu-j,0}_\cl$ 
symbols that coincide with the homogeneous components of the $\Lambda$-elliptic symbol 
$a\in S^{\mu,0}_\cl$ for $|\xi|\ge 1$. Moreover, we assume that $a^{(\mu)}(x,\xi)$ has no 
spectrum in $\Lambda^\prime$ for all $x,\xi\in\rz^n$. 

We let $\Sigma$ be a sector such that $\sigma\mapsto\sigma^\mu$ is a bijection from $\Sigma$ 
to the interior of $\Lambda^\prime$, and define 
 $$p^{(-\mu)}(x,\xi,\sigma)=\big(\sigma^\mu-a^{(\mu)}(x,\xi)\big)^{-1},$$ 
and then recursively 
  $$p^{(-\mu-i)}(x,\xi,\sigma)=\sum_{\substack{j+l+|\alpha|=i\\ j<i}}
    \frac{1}{\alpha!}(\partial_\xi^\alpha p^{(-\mu-j)})(D^\alpha_x a^{(\mu-l)})p^{(-\mu)},
    \qquad i\ge 1.$$
Then one can show that 
 $$p^{(-\mu)}\;\in\; S^{(-\mu,0);0}_\cl\,\cap\,S^{(0,0);-\mu}_\cl$$
and that $p^{(-\mu)}$ is homogeneous of degree $-\mu$ in $(\xi,\sigma)$ for $|\xi|\ge1$. 
Moreover,  
 $$p^{(-\mu-i)}\;\in\; S^{(-\mu-i,0);0}_\cl\,\cap\,S^{(\mu-i,0);-2\mu}_\cl,\qquad i\ge 1,$$
are homogeneous of degree $-\mu-i$ in $(\xi,\sigma)$ for $|\xi|\ge1$. 

\begin{remark}\rm\label{zero2.5}
Given the concrete form of the $p^{(-\mu-i)}$ it is straightforward to see that the limit symbols 
$p^{(-\mu)}_{(-\mu,j)}$ and $p^{(-\mu-i)}_{(-2\mu,j)}$ in the sense of \eqref{zero2.A.5} 
can be different from zero only if $j=k\mu$ for some $k\in\nz_0$.  
\end{remark}

We now define $p\in S^{(0,0),-\mu}_\cl\cap S^{(-\mu,0);0}_\cl$ by 
 $$p(x,\xi,\sigma)- p^{(-\mu)}(x,\xi,\sigma)\sim 
   \sum_{i=1}^\infty p^{(-\mu-i)}(x,\xi,\sigma).$$
Then $p$ is a weakly polyhomogeneous symbol and   
\begin{align*}
 p-p^{(-\mu)}\;\in\; S^{(-\mu-1,0);0}_\cl \,\cap\, S^{(\mu-1,0);-2\mu}_\cl. 
\end{align*}
Furthermore, the asymptotic summation can be performed in such a way that the property 
described in Remark \ref{zero2.5} carries over to $p-p^{(-\mu)}$. 

\begin{proposition}\label{zero2.6} 
$p$ is a parametrix of $\sigma^\mu-a(x,D)$, i.e. 
 $$r(x,\xi,\sigma):=1-(\sigma^\mu-a(x,\xi))\,\#\,p(x,\xi,\sigma)\;\in\;  
   S^{(-\infty,0);-\mu}_\cl.$$
The limit symbols $r_{(-\mu,j)}$, in the sense of \eqref{zero2.A.5}, can be different from 
zero only if $j=k\mu$ for some $k\in\nz_0$. 
\end{proposition}

{From} Proposition \ref{zero2.6} one derives that $(\sigma^\mu-a(x,D))^{-1}$ exists for all 
$\sigma\in\Sigma$ with sufficiently large absolute value, and 
 $$(\sigma^\mu-a(x,D))^{-1}=p(x,D,\sigma)+
   \sum_{j=1}^\infty p(x,D,\sigma)\,r(x,D,\sigma)^j. $$
For the $k$-th power of the resolvent we then obtain 
 $$(\sigma^\mu-a(x,D))^{-k}=p(x,D,\sigma)^k
   +\sum_{l=1}^\infty p(x,D,\sigma)\,r_l(x,D,\sigma),$$
with $p\,\#\,r_l\in S^{(-\infty,0);-(k+l)\mu}_\cl$. Moreover, 
 $$p(x,D,\sigma)^k=p^{(-\mu)}(x,D,\sigma)^k+p^\prime(x,D,\sigma)$$ 
with 
 $$p^\prime\sim\sum_{j=1}^\infty p_j^\prime,\qquad 
   p_j^\prime\in S^{(-k-j,0);0}_\cl\,\cap\,S^{(\mu-j,0);-(k+1)\mu}_\cl,$$
and the $p_j^\prime$ being homogeneous of degree $-k\mu-j$ in $(\xi,\sigma)$ for 
$|\xi|\ge1$. 

Applying now \eqref{zero2.B} or \eqref{zero2.B.5} to each of the terms, keeping in mind 
the vanishing of limit symbols described before, and inserting $\sigma=\lambda^{1/\mu}$ 
one obtains Theorem \ref{zero2.2}; for more details see the original proofs 
in \cite{GrSe}.

\section{Relation to Relative Determinants}\label{sec:zero3}

\subsection{Asymptotics of the Heat Trace}

Throughout this subsection let $\Lambda=\Lambda(\theta)$ for some $\theta<\frac{\pi}{2}$ 
and $A=a(x,D)$ with a symbol $a\in S^{\mu,0}_\cl$, $\mu\in\nz$, that is parameter-elliptic 
with respect to $\Lambda$. Moreover, we assume that the resolvent of $A$ exists in the whole 
sector $\Lambda$. 

\begin{theorem}\label{zero3.1}
The integral 
\begin{equation}\label{zero3.A}
 e^{-tA}=\frac{1}{2\pi i}\int_{\partial\Lambda}e^{-t\lambda}(\lambda-A)^{-1}\,d\lambda,
 \qquad t>0,
\end{equation}
defines a pseudodifferential operator with symbol belonging to 
$S^{-\infty,0}_\cl$ for each $t>0$. 
\end{theorem}
\begin{proof}
Let $b(x,D,\lambda)$ be the parametrix described in Theorem \ref{zero5.3} of the Appendix. 
Since 
 $$[\lambda]^{-2}\big\{(\lambda-A)^{-1}-b(x,D,\lambda)\big\}\in S^{-\infty,-\infty}$$
uniformly in $\lambda\in\Lambda$, we can replace in \eqref{zero3.A} the resolvent by 
$b(x,D,\lambda)$. Since there are symbols $b_j(x,\xi,\lambda)\in S^{-\mu-j,-j}$ such that 
 $$b(x,\xi,\lambda)-\sum_{j=0}^{N-1}b_j(x,\xi,\lambda)\;\in\; S^{-\mu-N,-N}$$
uniformly in $\lambda\in\Lambda$, we obtain
 $$e^{-tA}-\frac{1}{2\pi i}\sum_{j=0}^{N-1}
   \int_{\partial\Lambda}e^{-t\lambda}b_j(x,D,\lambda)\,d\lambda\;\in\;S^{-\mu-N,-N}.$$
It remains to show that 
  $$\int_{\partial\Lambda}e^{-t\lambda}b_j(x,D,\lambda)\,d\lambda\;\in\,S^{-\infty,0}_\cl,
    \qquad j\in\nz_0.$$ 
However, using the fact -- cf.\ \cite{MSS} for details -- that $b_j$ is a finite linear combination 
of finite products of derivatives of the symbol $\wt{a}(x,\xi)$ (see \eqref{zero5.A} for the 
definition of $\wt{a}$) and of factors $(\lambda-\wt{a}(x,\xi))^{-1}$, one can show the validity 
of the uniform estimates 
 $$|\partial_\xi^\alpha\partial_x^\beta\partial_\lambda^\gamma b_j(x,\xi,\lambda)|
     \le C (|\lambda|+[\xi]^\mu)^{-1-|\gamma|}[x]^{-|\beta|}[\xi]^{-j-|\alpha|}$$
for any order of derivatives. Then the claim follows from integration by parts. 
\end{proof}

By the previous theorem we can apply the regularized trace to the heat operator. 
It will be important to know that one can take the trace under the integral in \eqref{zero3.A}. 
To achieve this we need the following result:

\begin{proposition}\label{zero3.2}
For $k\in\nz$ let $a_k(x,\xi,\lambda)$ be the symbol of $(\lambda-A)^{-k}$, the $k$-th 
power of the resolvent of $A$. Then the following functions are bounded: 
 $$\lambda\mapsto[\lambda]^ka_k(x,\xi,\lambda):\Lambda\to S^{0,0}_{\cl(x)},\qquad 
     \lambda\mapsto a_k(x,\xi,\lambda):\Lambda\to S^{-\mu k,0}_{\cl(x)}.$$ 
\end{proposition}
\begin{proof}
Clearly we may assume $k=1$. If $b(x,\xi,\lambda)$ is as in Theorem \ref{zero5.3}, 
it is enough to consider $b(x,D,\lambda)$ instead of $(\lambda-A)^{-1}$. By the structure 
of $b(x,\xi,\lambda)$ (cf.\ the beginning of the proof of Theorem \ref{zero3.1}), the 
proof reduces to showing the corresponding properties for $(\lambda-\wt{a}(x,\xi))^{-1}$. 
To this end observe that if $\phi$ is the diffeomorphism of Proposition \ref{zero5.2.5}, 
then
\begin{equation}\label{zero3.B}
 |\partial_y^\beta\partial_\xi^\alpha\partial^j_\lambda
 \big(\lambda-\wt{a}(\phi(y),\xi)\big)^{-1}|
 \le C_{\alpha\beta k}(|\lambda|+[\xi]^\mu)^{-1-j}[\xi]^{-|\alpha|}.
\end{equation}
In fact, this is true by chain rule, since 
 $$|\partial_y^\beta\partial_\xi^\alpha \wt{a}(\phi(y),\xi)|\le 
   C_{\alpha\beta}[\xi]^{\mu-|\alpha|}$$
for $\wt{a}(x,\xi)$ is a classical symbol in $x$, and 
 $$|\big(\lambda-\wt{a}(\phi(y),\xi)\big)^{-1}|
   \le C(|\lambda|+[\xi]^\mu)^{-1}$$
by the parameter-ellipticity of $\wt{a}$. Clearly \eqref{zero3.B} 
implies boundedness with respect to 
$\lambda\in\Lambda$ of 
 $$[\lambda](\lambda-\wt{a}(x,\xi))^{-1}\;\in\;S^{0,0}_{\cl(x)},\qquad 
   (\lambda-\wt{a}(x,\xi))^{-1}\;\in\;S^{-\mu,0}_{\cl(x)}.$$
\end{proof}

Using integration by parts and choosing $k>1+\frac{n+1}{\mu}$, we have 
\begin{equation}\label{zero3.C}
 \TR e^{-tA}=t^{-k}\frac{(-1)^{k}k!}{2\pi i}\int_{\partial\Lambda}e^{-t\lambda}\,
 \TR (\lambda-A)^{-k-1}\,d\lambda;
\end{equation}
in fact, by Proposition \ref{zero3.2}, we know that the symbol of 
$[\lambda]^{2}(\lambda-A)^{-k}$ belongs to the space $S^{-n-1,0}_{\cl(x)}$,
uniformly in $\lambda$. Hence, the continuity of the trace functional allows to pull $\TR$ 
under the integral. Due to the holomorphicity of the integrand, we can 
modify the path of integration, $\partial\Lambda$, to a parametrization of the boundary 
of $\Lambda\cup\{|\lambda|\le\eps\}$ for sufficiently small $\eps>0$. 
We shall use this fact frequently without indicating it explicitly. 

{From} \eqref{zero3.C} it follows that $\TR e^{-tA}$ is holomorphic in a sector of the 
form $\{t=re^{i\varphi}\st r>0,\,|\varphi|<\theta_0\}$ for $0<\theta_0\le \pi/2-\theta$, 
decays there exponentially for $|t|\to\infty$, and when $|t|\to0$ it is $O(|t|^{-k})$ 
for any integer $k>1+\frac{n+1}{\mu}$. 

We shall now show that $\TR e^{-tA}$ has an asymptotic expansion for $t\to0$ which 
is related to the pole structure of the regularized zeta function $\zeta_A$: If $\Gamma$ 
denotes the standard Gamma function then, according to Theorem \ref{zero1.7}
\begin{equation}\label{zero3.D}
 \Gamma(z)\zeta_A(z)\sim\sum_{k=0}^\infty\sum_{l=0}^1 
 c_{kl}\Big(z-\frac{n-k}{\mu}\Big)^{-(l+1)},
\end{equation}
i.e.\ $\Gamma(z)\zeta_A(z)$ is a meromorphic function on $\cz$ with poles in 
$\frac{n-k}{\mu}$, $k\in\nz_0$. Furthermore, $c_{k1}=0$ when $k=n$ or 
$k-n\notin\mu\nz$. 

\begin{theorem}\label{zero3.3}
The generalized trace of $e^{-tA}$ has the following asymptotic expansion: 
 $$\TR e^{-tA}\sim 
   \sum_{k=0}^\infty\sum_{l=0}^1 (-1)^lc_{kl}\,t^{-\frac{n-k}{\mu}}\,\log^l t,\qquad 
   t\longrightarrow 0+,$$
where the coefficients $c_{kl}$ are those from \eqref{zero3.D}. 
\end{theorem}
\begin{proof}
By Proposition 5.1.3${}^\circ$ in \cite{GrSe2} applied to $r(\lambda)=\TR(\lambda-A)^{-k-1}$, 
we obtain that $t^k\frac{(-1)^k}{k!}\TR e^{-tA}$ is the inverse Mellin transform of 
 $$\Gamma(z)\frac{1}{2\pi i}\int_{\partial\Lambda}\lambda^{-z}\TR(\lambda-A)^{-k-1}\,d\lambda$$
(this function initially is defined as a holomorphic function in $\re z>0$ and then extends 
meromorphically to $\cz$). Standard properties of the Mellin transform show that 
$\TR e^{-tA}$ is the inverse Mellin transform of 
 $$\Gamma(z-k)\frac{(-1)^k k!}{2\pi i}\int_{\partial\Lambda}
   \lambda^{k-z}\TR(\lambda-A)^{-k-1}\,d\lambda.$$ 
Using the identity $\Gamma(z)=(z-1)\cdot\ldots\cdot(z-k)\Gamma(z-k)$, pulling the trace 
in front of the integral (which is justified by Proposition \ref{zero3.2}), and using 
integration by parts, we see that $\TR e^{-tA}$ is the inverse Mellin transform of 
$$\TR\Big(\Gamma(z)\frac{1}{2\pi i}
  \int_{\partial\Lambda}\lambda^{-z}(\lambda-A)^{-1}\,d\lambda\Big)=
  \TR(\Gamma(z)A^{-z})=\Gamma(z)\zeta_A(z).$$
{From} Proposition 5.1.2${}^\circ$  in \cite{GrSe2} we can derive the desired asymptotic 
behavior of $\TR e^{-tA}$ as $t\to0$ in the sector 
$\{t=re^{i\varphi}\st r\ge0,\,|\varphi|<\theta_0\}$ from the known pole structure of its 
Mellin transform $\Gamma(z)\zeta_A(z)$, provided the following decay property holds: 
For each real $a>0$ and $\delta<\theta_0$,
\begin{equation}\label{zero3.E}
 |\Gamma(z)\zeta_A(z)|\le C_{a,\delta}e^{-\delta|\im z|},\qquad 
 |\im z|\ge1,\quad |\re z|\le a.
\end{equation}
To this end we write 
 $$\zeta(z-k):=\frac{(-1)^k}{k!}(z-1)\cdot\ldots\cdot(z-k)\zeta_A(z)=
   \frac{1}{2\pi i}\int_{\partial\Lambda}\lambda^{k-z}\TR(\lambda-A)^{-k-1}\,d\lambda.$$
Due to Theorem \ref{zero3.1} we can apply Proposition 2.9 of \cite{GrSe2} with  
$f(\lambda)=\TR(\lambda-A)^{-k-1}$ to obtain that 
 $$\Big|\frac{\pi\zeta(z-k)}{\sin\pi(z-k)}\Big|\le C\,e^{-\delta|\im z|}$$
for any $\delta<\pi-\theta$, uniformly for $|\im z|\ge 1$ and $|\re z|\le a$.
Since $\displaystyle \frac{\pi}{s\sin\pi s}=-\Gamma(-s)\Gamma(s)$, 
$\Gamma(z)=(z-1)\cdot\ldots\cdot(z-k)\Gamma(z-k)$, and 
 $$|\Gamma(s)|\ge c_a\,|\im s|^{\re s-\frac{1}{2}}e^{-\frac{\pi}{2}|\im s|},$$
cf. \cite[Chapter VII, Exercise \S2.3]{Bourbaki}, we arrive at 
 $$|\Gamma(z)\zeta_A(z)|\le C\,|\im z|^{a-k-\frac{1}{2}}e^{-(\delta-\frac{\pi}{2})|\im z|}$$
for any $\delta<\pi-\theta$, uniformly for $|\im z|\ge 1$ and $|\re z|\le a$. 
This clearly implies \eqref{zero3.E} for $\theta_0=\frac{\pi}{2}-\theta$.
\end{proof}

\subsection{Relative Zeta Functions and Determinants}

Let us now consider two positive self-adjoint operators $A=a(x,D)$ and $A_0=a_0(x,D)$  
with symbols $a,\ a_0 \in S^{\mu,0}_\cl$, $\mu\in\nz$.

\begin{lemma}\label{zero3.4}
Let $t>0$ and $e^{-tA}-e^{-tA_0}$ be a trace class operator. Then 
 $$\TR (e^{-tA} -e^{-tA_0})={\rm tr}\;(e^{-tA} -e^{-tA_0}).$$
\end{lemma}
\begin{proof} From Theorem \ref{zero3.1} we know that both $e^{-tA}$ and $e^{-tA_0}$ are 
pseudodifferential operators with symbol in $S^{-\infty,0}_\cl$. In particular, 
$e^{-tA}-e^{-tA_0}$ is a trace class operator with a continuous 
integral kernel $k(x,y)$. Hence its trace is given by $\displaystyle\int k(x,x)\,dx$ 
(cf., for example, Section 19.3 of  \cite{Horm}). 
On the other hand $k(x,x)$ is a zero order classical symbol and the generalized 
trace of $e^{-tA}-e^{-tA_0}$ equals $\displaystyle\fpint k(x,x)\,dx$. 
Now the claim follows, since we observed in Section \ref{sec:zero1} 
that the finite-part integral coincides with the usual integral, if the integrand is in 
$L_1(\rz^n)$. 
\end{proof}

According to Theorem \ref{zero3.3} ${\rm tr}\;(e^{-tA} -e^{-tA_0})$ has an 
asymptotic expansion as $t\to 0^+$ and is rapidly decreasing as $t\to\infty$. 
Following \cite{Mull} the  \textit{relative zeta function} of $(A,A_0)$ is 
defined as follows:
\begin{equation}\label{zero3.F}
 \zeta(z;A,A_0):=   \frac{1}{\Gamma(z)}\int_0^\infty 
 t^{z-1} {\rm tr}\;(e^{-tA} -e^{-tA_0})\;dt
\end{equation}
In the proof of Theorem \ref{zero3.3} we saw that $\TR e^{-tA}$ and 
$\TR e^{-tA_0}$ are the inverse Mellin transforms of $\Gamma(z)\zeta_A(z)$ 
and $\Gamma(z)\zeta_{A_0}(z)$, respectively, whence, from \eqref{zero3.F}, 
we obtain that 
$$
\zeta(z;A,A_0)= \zeta_A(z)-\zeta_{A_0}(z), 
$$
i.e.\ the relative zeta function coincides with the difference of the regularized 
zeta functions introduced in Section \ref{sec:zero1.3}. The previous expression 
and the results of Section \ref{sec:zero1.3} allow us to recover the properties
 of the regularized zeta function proven in \cite{Mull} for the kind of operators 
we are considering. In particular, we obtain that $\zeta(z;A,A_0)$ is 
holomorphic at $z=0$, and we can conclude: 

\begin{theorem}\label{zero3.5}
Let $A=a(x,D)$ and $A_0=a_0(x,D)$ be positive and self-adjoint with symbols 
$a,\ a_0 \in S^{\mu,0}_\cl$, $\mu\in\nz$, and assume that $e^{-tA}-e^{-tA_0}$ 
is a trace class operator for any $t>0$. Then the relative determinant 
 $$\mathrm{det}(A,A_0):=\exp\Big(-\frac{d}{dz}\Big|_{z=0}\zeta(z;A,A_0)\Big)$$
satisfies 
 $$\mathrm{det}(A,A_0)= \frac{{\rm det } A}{{\rm det} A_0},$$
where the determinants on the right-hand side are defined in 
Definition {\rm\ref{zero1.9}}. 
\end{theorem}

\section{Regularized Zeta Function and Determinant in Case $m\not=0$}\label{sec:zero6}

In this section we extend our definition of determinants to operators $a(x,D)$ with symbol 
$a\in\smmu_\cl$, $m,\mu>0$. We shall assume that $a(x,D)$ is $\Lambda$-elliptic and has 
no spectrum in the sector $\Lambda$ of \eqref{zeroi.C}.

\begin{definition}\rm\label{zero6.1}
For $m,\mu\in\cz$ such that $\mu,\,m\not\in\gz_{\ge-n}$ we define 
 $$\TR:S^{\mu,m}_\cl\longrightarrow\cz,\quad a(x,D)\mapsto\fpiint a(x,\xi)\,dx\dbar\xi.$$
\end{definition}

Note that if $\re m<-n$ and $\re\mu<-n$ then $\TR a(x,D)$ is the usual trace of the 
trace class operator $a(x,D)$ in $\mathcal L(L_2(\rz^n))$. 
The proposition, below, follows from Proposition \ref{zero1.3}. 

\begin{proposition}\label{zero6.2}
Let $a(x,\xi,z)\in S^{\mu(z),m(z)}_\cl$ be a holomorphic family on $U\subseteq\cz$ with 
$\mu\not\equiv\mathrm{const}$,  $m\not\equiv\mathrm{const}$, and $A(z)=a(x,D,z)$. Let 
 $$\calP=\{z\in U\st \mu(z)+n\in\nz_0\text{\rm\ or } m(z)+n\in\nz_0\}.$$
Then $z\mapsto\TR A(z)$ is meromorphic  on $U$ with poles at most 
in $\calP$. If $z_0\in\calP$ and both $\mu(z_0)\in\gz$ and $m(z_0)\in \gz$ then, 
near $z_0$, 
\begin{align*}
 \TR A(z)\equiv&\,
   -\frac{(2\pi)^{-n}}{(m(z_0)-m(z))(\mu(z_0)-\mu(z))}\int_{|\xi|=1} \int_{|x|=1}
   a^{(\mu(z)-\mu(z_0)-n)}_{(m(z)-m(z_0)-n)}(z)\,d \sigma(x) d \sigma(\xi)+\\
&+\frac{1}{m(z_0)-m(z)}
   \fpint\!\!\!\int_{|x|=1} a_{(m(z)-m(z_0)-n)}(z)\, d \sigma(x) \dbar \xi+\\
&+\frac{1}{m(z_0)-m(z)}\sum_{k=0}^{n+\mu(z_0)-1}\frac{(2\pi)^{-n}}{n+\mu(z)-k}
   \int_{|\xi|=1}\int_{|x|=1}
   a^{(\mu(z)-k)}_{(m(z)-m(z_0)-n)}(z)\,d\sigma(x)d\sigma(\xi)+\\
&+\frac{1}{\mu(z_0)-\mu(z)}
   \fpint\!\!\!\int_{|\xi|=1} 
   a^{(\mu(z)-\mu(z_0)-n)}(z)\, d \sigma(\xi) \dbar x+\\
&+\frac{1}{\mu(z_0)-\mu(z)}\sum_{j=0}^{n+m(z_0)-1}\frac{(2\pi)^{-n}}{n+m(z)-j}
   \int_{|\xi|=1}\int_{|x|=1}
   a^{(\mu(z)-\mu(z_0)-n)}_{(m(z)-j)}(z)\,d\sigma(x)d\sigma(\xi)
\end{align*}
modulo a function which is holomorphic in $z_0$. If $\mu(z_0)\in\gz$ and $m(z_0)\notin\gz$, 
the corresponding formula is obtained from the above one by deleting the first, second, and 
third line, and taking in the last line the summation in $j$ from $j=0$ up to 
$j=\entier{\re m(z_0)+n}$. Analogously, when $\mu(z_0)\notin\gz$ and $m(z_0)\in\gz$, the 
corresponding formula is obtained by deleting the first, fourth, and fifth line and taking in 
the third line the summation in $k$ from $k=0$ up to $j=\entier{\re\mu(z_0)+n}$.
\end{proposition}

\begin{proposition}\label{zero6.3}
Let $A_j=a_j(x,D)$, $j=0,1$, with $a_j\in S^{\mu_j,m_j}_{\cl}$ and  
$\mu_0+\mu_1,m_0+m_1\notin\gz_{\ge-n}$. Then $\TR(A_0A_1)=\TR(A_1A_0)$. 
\end{proposition}
\begin{proof}
First assume that the numbers $\mu_0/2$, $m_0/2$, $\mu_0/2+\mu_1$, and $m_0/2+m_1$ 
have real parts smaller than $-n$. Let $D$ be the operator with symbol 
$[x]^{m_0/2}[\xi]^{\mu_0/2}$. Then 
 $$\TR(A_0A_1)=\tr(A_0A_1)=\tr((A_0D^{-1})(DA_1))=\tr(D(A_1A_0D^{-1}))
    =\tr(A_1A_0)=\TR(A_1A_0),$$
using the commutativity of the usual trace on trace class operators. In other words, 
the result holds true if one of the operators has orders with sufficiently negative real 
parts. 

For the general case choose a positive integer $L$ such that $\mu:=\mu_0+L>0$ and 
$m:=m_0+L>0$. Choosing the constant $c$ in $p(x,\xi):=c\spk{x}^m\spk{\xi}^\mu$ 
sufficiently large we can achieve that both $P:=p(x,D)$ and $Q:=P+A_0$ are 
$\Lambda$-elliptic operators not having spectrum in $\Lambda$. Now consider the 
holomorphic families 
 $$B_0(z)=(Q^z-P^z)A_1,\qquad B_1(z)=A_1(Q^z-P^z).$$
If $q(x,\xi)$ is the symbol of $Q$ then, by construction, $q^{(\mu-j)}=p^{(\mu-j)}$ 
and $q_{(m-k)}=p_{(m-k)}$ for $0\le j,k\le L-1$. This implies that 
 $$B_0(z),B_1(z)\in S^{mz+m_1-L,\mu z+\mu_1-L}_\cl$$
$($as follows from Theorem 3.2 and Proposition 3.1 of \cite{MSS}$)$. 
By Proposition \ref{zero6.2} both $\TR B_0(z)$ and $\TR B_1(z)$ are meromorphic 
functions on the whole plane that are holomorphic in $z=1$. By the first part of the proof, 
they coincide whenever $\re z$ is sufficiently negative. Hence we can conclude that 
$\TR B_0(1)=\TR B_1(1)$, i.e. $\TR(A_0A_1)=\TR(A_1A_0)$. 
\end{proof}

Applying Proposition \ref{zero6.2} to the holomorphic family of complex powers of a 
$\Lambda$-elliptic operator $a(x,D)$ that has no spectrum in $\Lambda$, and using 
that $A^0=1$, we obtain that 
\begin{equation}\label{zero6.A}
 \TR a(x,D)^{-z}\equiv
 -\frac{(2\pi)^{-n}}{m\mu z^2}\int_{|\xi|=1} \int_{|x|=1}
 a^{(-\mu z-n)}_{(-m z-n)}(z)\,d \sigma(x) d \sigma(\xi)
\end{equation}
near $z=0$. Therefore $\TR a(x,D)^{-z}$ has possibly a pole in $z=0$ which is simple, 
since the integral term tends to zero for $z\to0$.

\begin{definition}\rm\label{zero6.4}
Let $a\in S^{\mu,m}_\cl$, $m,\mu>0$, be parameter-elliptic with respect to $\Lambda$, 
and assume that the resolvent of $A:=a(x,D)$ exists in the whole sector $\Lambda$. 
Then define the zeta function of $A$ as the meromorphic function 
 $$\zeta_A(z)=\TR A^{-z},$$
and the determinant of $A$ by 
 $${\rm det}\,A:=\exp\Big(-\mathrm{res}\big|_{z=0}\frac{\zeta_A(z)}{z^2}\Big).$$
\end{definition}

Note that \textit{if} the zeta function of $A=a(x,D)$ is holomorphic in $z=0$, then the 
above defined determinant coincides with the usual zeta regularized determinant, 
i.e.\ $\det\,A=\exp(-\zeta^\prime_A(0))$. 

To give a general criterion which implies that the zeta function has no pole in $z=0$
let us recall the notion of symbols with parity: We say that a symbol $a\in\smmu_\cl$ 
has \textit{even-even parity in $\xi$}, if
 $$a^{(\mu-k)}(x,-\xi)=(-1)^{\mu-k}a^{(\mu-k)}(x,\xi),\qquad 
  k\in\nz_0, x\in \rz^n, \xi\not=0.$$
We say that $a$ has \textit{even-even parity in 
$x$} if the corresponding relations hold for all $a_{(m-j)}$, $j\in\nz_0$, for all 
$x\not=0$ and $\xi\in\rz^n$. Similarly, $a$ has \textit{even-odd parity in} $\xi$, if 
$$a^{(\mu-k)}(x,-\xi)=(-1)^{\mu-k-1}a^{(\mu-k)}(x,\xi),\qquad 
  k\in\nz_0, x\in \rz^n, \xi\not=0$$
and \textit{even-odd parity in} $x$ if the corresponding relations hold for all 
$a_{(m-j)}$, $j\in\nz_0$, for all 
$x\not=0$ and $\xi\in\rz^n$. 

\begin{theorem}\label{zero6.5} 
Let $n$ be odd, $A=a(x,D)$ for $a\in S^{\mu,m}_\cl$ with $\mu,m>0$ be 
$\Lambda$-elliptic without spectrum in  $\Lambda$. Let either $\mu$ be an even integer 
and $a$ have even-even parity in $\xi$, or $m$ be an even integer and $a$ have 
even-even parity in $x$. Then the zeta function of $A$ is holomorphic in $z=0$.
\end{theorem}
\begin{proof}
Let $a(x,\xi,z)$ denote the symbol of $A^{-z}$. Let $\mu$ be an even integer and 
assume that $a(x,\xi)$ is even-even in $\xi$. The explicit formulas for the homogeneous 
components of the symbols of the complex powers derived in Section 3.4 of \cite{MSS} 
show that
 $$a^{(-\mu z-k)}(x,-\xi,z)=(-1)^{k}a^{(-\mu z-k)}(x,\xi,z),\qquad
   a^{(-\mu z-k)}_{(-mz-j)}(x,-\xi,z)=(-1)^{k}a^{(-\mu z-k)}_{(-mz-j)}(x,\xi,z)$$
for all $j,k\in\nz_0$ and all $z\in\cz$, i.e.\ all the $\xi$-homogeneous components are 
odd functions in $\xi$. Hence the right-hand side in \eqref{zero6.A} equals zero, 
i.e. the zeta function is holomorphic near zero. 
This concludes the proof of the first case. The other case is analogous.
\end{proof}

\begin{remark}In a similar way one can establish the regularity of the zeta function 
in zero in case $n$ is odd, $\mu$ is an odd integer and 
$a$ has even-odd parity in $\xi$, or $m$ is an odd integer and $a$ has 
even-odd parity in $x$. 
\end{remark}

\subsection{Relation to Other Trace Functionals}

Let 
 $$\calD=\Big\{A=a(x,D)\mid 
   a\in \smmu_\cl\text{ with }\mu,m\in\cz\setminus\gz_{\ge-n}\Big\}.$$
A regularized trace functional on $\calD$ is a map $\tau:\calD\to\cz$ with the following three 
properties:
\begin{itemize}
 \item[(1)] $\tau$ is linear in the sense that 
  $\tau(\lambda_0A_0+\lambda_1A_1)=\lambda_0\tau(A_0)+\lambda_1\tau(A_1)$ whenever 
  $A_0,A_1$ and $\lambda_0A_0+\lambda_1A_1$ belong to $\calD$. 
 \item[(2)] $\tau(A_0A_1)=\tau(A_1A_0)$ 
  whenever $A_0A_1$ and $A_1A_0$ belong to $\calD$. 
 \item[(3)]   $\tau(A)=\tr A$ whenever $A$ is a trace class operator.
 \end{itemize}
By Proposition \ref{zero6.3} $\TR$ is such a regularized trace functional. 
The following theorem clarifies the relation between an arbitrary regularized trace functional and $\TR$: 

\begin{theorem}\label{thm:unique}
Let $\tau$  satisfy $(1)-(3)$. Then$:$ 
\begin{itemize}
 \item[a$)$] $\tau=\TR$ on any $S^{\mu,m}_\cl$ with $\mu\not=-n-1$ and $m\not=-n-1$. 
 \item[b$)$] On $S^{-n-1,m}_\cl$ the difference $\tau-\TR$ depends only on the homogeneous 
   components $a^{(-n-1)}_{(m-k)}$ with $k=0,\ldots,\entier{\re m+n}$.  
 \item[c$)$] On $S^{\mu,-n-1}_\cl$ the difference $\tau-\TR$ depends only on the homogeneous 
   components $a_{(-n-1)}^{(\mu-j)}$ with $j=0,\ldots,\entier{\re \mu+n}$.  
\end{itemize}
\end{theorem}

In \cite{MSS2} the authors have shown that the analog of conditions $(1)-(3)$ 
determines trace functionals uniquely for classical pseudodifferential operators on closed
manifold: The only regularized trace is the Kontsevich-Vishik trace. 
See also Lesch and Neira \cite{LN} for an extension.
The current setting is more complex due 
to the double grading with respect to the orders $\mu$ and $m$, respectively. 

\begin{proof} 
Let $a\in \smmu_\cl$ with $\mu,m\in\cz\setminus\gz_{\ge-n}$ and let $\kappa(x)$ and 
$\chi(\xi)$ be two 0-excision functions which are $1$ for $|x|\ge1$ and $|\xi|\ge1$, 
respectively. For convenience we use the short-hand notation $\tau(a)=\tau(a(x,D))$. 

a$)$ Let $\mu\not=-n-1$ and $m\not=-n-1$. We write   
 $$a=\sum_{j=0}^N\chi a^{(\mu-j)}+\sum_{k=0}^N\kappa a_{(m-k)}
   -\sum_{j,k=0}^N\kappa\chi a^{(\mu-j)}_{(m-k)}+r_N$$
with a remainder $r_N\in S^{\mu-N-1,m-N-1}$, cf.\ Lemma 2.1 in \cite{MSS}. 
Defining 
 $$a^j=\chi\Big(a^{(\mu-j)}-\sum_{k=0}^N\kappa a^{(\mu-j)}_{(m-k)}\Big),
   \quad 
   a_k=\kappa\Big(a_{(m-k)}-\sum_{j=0}^N\chi a^{(\mu-j)}_{(m-k)}\Big),
   \quad
   a^j_k=\kappa\chi a^{(\mu-j)}_{(m-k)},
$$
we can rewrite this identity as 
 $$a=\sum_{j=0}^Na^j+\sum_{k=0}^N a_k+
     \sum_{j,k=0}^N a^j_k+r_N.$$
Choosing $N$ large enough, $r_N(x,D)$ is trace class and therefore 
$\tau(r_N)=\TR (r_N)$. Moreover, 
$a^j\in S^{\mu-j,m-N-1}$ is positively homogeneous of degree $\mu-j$ in 
$|\xi|\ge1$. Using Euler's identity for homogeneous functions we thus obtain 
that $a^j$ differs from 
 $$b^j(x,\xi):=\frac{1}{n+\mu-j}\sum_{i=1}^n\partial_{\xi_i}(a^j(x,\xi)\xi_i)$$
by a symbol which has order $m-N-1$ in $x$ and compact support in $\xi$, hence is trace class. Since 
 $$b^j(x,D)=\frac{-i}{n+\mu-j}\sum_{i=1}^n [x_i,a^j(x,D)D_{x_i}]$$
is a sum of commutators, $\tau(b^j)=0$. Hence $\tau(a^j)=\TR (a^j)$ for any 
$j$.\footnote{In case $\mu=-n-1$ 
this argument breaks down for $j=0$, because then $x_ia^0(x,\xi)\xi_i$ has $\xi$-order 
$-n$, hence is not an element of $\calD$.}  
Analogously, $a_k(x,D)$ differs from 
 $$b_k(x,D)=\frac{-i}{n+m-k}\sum_{i=1}^n [D_{x_i},x_ia_k(x,D)]$$
by a trace class operator, thus also $\tau(a_k)=\TR (a_k)$ for any $k$. 
Finally, $a^j_k$ differs from 
 $$b^j_k(x,\xi):= \frac{1}{n+\mu-j}\sum_{i=1}^n
     \partial_{\xi_i}\big(a^j_k(x,\xi)\xi_i\big)$$
by a remainder $r^j_k(x,\xi)$ which is compactly supported in $\xi$ and positively 
homogeneous of degree $m-k$ in $|x|\ge1$. Therefore $r^j_k$ differs from 
 $$c^j_k(x,\xi):= \frac{1}{n+m-k}\sum_{i=1}^n\partial_{x_i}
   \big(x_i r^j_k(x,\xi)\big)$$
by a remainder which is compactly supported in $(x,\xi)$, hence trace class. 
Applying $\tau$ to $b^j_k(x,D)$ and $c^j_k(x,D)$ gives zero, since both these 
operators are sums of commutators. We conclude that also $\tau(a^j_k)=\TR (a^j_k)$. 

b$)$ Now let $\mu=-n-1$. If $m<-n$ then $a(x,D)$ is trace class and $\tau(a)=\TR(a)$. 
Thus we may assume that 
$m>-n$ and that $m$ is not integer. With $N:=\entier{\re m+n}+1$ we then proceed as in a$)$ 
and obtain that 
 $$(\tau-\TR)(a)=(\tau-\TR)\Big( a^0+\sum_{k=0}^N a^0_k\Big)
    =(\tau-\TR)(\chi a^{(\mu)}).$$
It follows that 
 $$(\tau-\TR)(a)= (\tau-\TR)\Big(\sum_{k=0}^N \kappa\chi a^{(\mu)}_{(m-k)}\Big),$$
since the difference of $\chi a^{(\mu)}$ with the last sum is of trace class. 
This shows b$)$, and c$)$ is verified analogously. 
\end{proof}


\section{Regularized Trace-Functionals and Determinants on Manifolds with Ends}
The results of the previous sections extend naturally to an $n$-dimensional 
SG-compatible manifold $\mathcal M$ with ends, cf. \cite{Schr0}. 
We shall follow here the set-up devised by Battisti and Coriasco in 
\cite[Section 4]{BaCo1}, which makes it possible to work with classical symbols:
The manifold $\mathcal M$ consists of a compact manifold $\mathcal M_0$ with boundary  
to which a finite number of ends $E_1, E_2, \ldots $ is attached. 
Topologically, each end $E_j$ has the structure of a half-cylinder, 
$]1,\infty[\,\times C_j$, where $C_j$ is a compact 
$(n-1)$-dimensional manifold. In \cite{BaCo1}, all $C_j$ are assumed to be
spheres, which is not necessary.

We assume that $C_j$ is covered by finitely many charts ${C}_{jk}$ homeomorphic 
to open subsets $U_{jk}$ of the sphere $\mathbb S^{n-1}$ and that there are coordinate
homeomorphisms $\kappa_{jk}\colon\,]1,\infty[\, \times C_{jk}\to E_{jk}= \{x\in \rz^n\st x/|x|\in U_{jk}, |x|>1\}$  such that the
transition maps $\kappa_{jk}\kappa_{jl}^{-1}$ are of the form 
$\Phi: x\mapsto |x|\varphi(x/|x|)$ with a diffeomorphism $\varphi$ of the corresponding open
subsets of $\mathbb S^{n-1}$. They thus satisfy the relation 
\begin{eqnarray}\label{chco}
\Phi(\lambda x ) = \lambda \Phi(x), \quad\lambda\ge 1.
\end{eqnarray}

We can then introduce the spaces 
$\mathcal S(\mathcal M)$ and $\mathcal S(\mathcal M\times\mathcal M)$ 
of rapidly decreasing functions on $\mathcal M$ and $\mathcal M\times\mathcal M$, respectively.
We call an operator $R: \cicomp(\mathcal M)\to {\mathcal C}^\infty(\mathcal M)$ regularizing, provided 
that its integral kernel with respect to the above coordinates is in 
$\mathcal S(\mathcal M\times\mathcal M).$

On each $E_j$, we introduce a special partition of unity: We pick a partition of unity 
$\{\psi_{jk}\}$ on $C_j$, subordinate to the cover $C_{jk}$, and define $\{\tilde\psi_{jk}\}$ on 
$E_j$ by $\tilde \psi_{jk}(x) = \psi_{jk}(x/|x|)$. Similarly we choose functions $\eta_{jk}\in \cicomp(C_{jk})$ 
with $\psi_{jk}\eta_{jk}=\psi_{jk}$ and let $\tilde \eta_{jk}(x) =\eta_{jk}(x/|x|)$. 
Changing the notation we then obtain a finite partition of unity $\{\psi_l\}$ and a corresponding 
set of cut-off functions $\{\eta_l\}$ on $\mathcal M$ with $\psi_l\eta_l=\psi_l$ 
by modifying the $\tilde \psi_{jk}$ and $\tilde \eta_{jk}$ near $|x|=1$ and picking corresponding
functions on $\mathcal M_0$.  

We call an operator $A:\cicomp(\mathcal M)\to \ci(\mathcal M)$ a pseudodifferential operator with local 
symbols in $S^{\mu,m}_{\cl}$, provided that for each $l$, the operator $\psi_l A \eta_l$ is - in local 
coordinates - a pseudodifferential operator with symbol in $S^{\mu,m}_{\cl}$, and $\psi_lA(1-\eta_l)$ 
is regularizing. This is well-defined in view of \eqref{chco}. The definition extends to the case of operators 
acting on sections of vector bundles with transition functions of corresponding type. 

Parameter-ellipticity -- as defined in Section  \ref{sec:zero5} --
allows the representation of  the resolvent as a parameter-depending parametrix and the
construction of  complex powers along the lines of \cite{Seel}. 
The complex powers are again classical SG-pseudodifferential operators as a consequence 
of the explicit formulas for the local symbols obtained in \cite{MSS} and the specific form 
of the changes of coordinates. For more details see also  \cite{BaCo1}. 

Given a pseudodifferential operator $A$ on $M$ with local symbols in $S^{\mu,m}_{\cl}$, 
let $K_A(x,y)$ be the distributional kernel 
and $\omega_{KV}(x)$ the Kontsevich-Vishik density associated to $K_A$ whose construction
was recalled in Section \ref{sec:zero1.2}. 
For $m,\mu\notin \gz_{\ge -n}$ we can then let 
\begin{eqnarray}\label{TR}
\TR(A) = \lim_{R\to\infty} \int_{|x|\le R} \omega_{KV}(x). 
\end{eqnarray}
Here $\int_{|x|\le R}$ means integration over the set of all elements of $ \mathcal M$
which either belong to $\mathcal M_0$ or have a local coordinate $x$ on an end with 
$|x|\le R$. The special form of the changes of coordinates along the ends makes 
\eqref{TR} well-defined.
It follows from the considerations in Section \ref{sec:zero1.2} that this coincides with the 
previous definition for the case $\mathcal M=\rz^n$.

\section{Appendix: \ SG-pseudodifferential Operators, Parameter-ellipticity, and Complex Powers}\label{sec:zero5}

The calculus of classical SG-pseudodifferential operators has been developed by
Hirschmann  \cite{Hirs}.
We summarize here the basic results including material on parameter-ellipticity, resolvent, 
and complex powers. For details we refer the reader to \cite{MSS} and  \cite{EgSc}.

In the sequel, let $[\,\cdot\,]$ denote a smooth, positive function on $\rz^n$ 
that coincides with the Euclidean norm outside the unit ball. 

\subsection{Symbols with Values in a  Fr\'echet Space}\label{sec:zero5.0}

Let $E$ be a Fr\'echet space and $\nu\in\rz$. 
By definition, $S^{(\nu)}(\rz^n,E)$ is the space of smooth functions $a:\rz^n\setminus\{0\}\to E$ 
that are positively homogeneous of degree $\nu$, i.e., 
 $$a(ty)=t^\nu a(y)\qquad\forall\;t>0\quad\forall\;y\not=0.$$
We shall denote by $S^\nu(\rz^n,E)$ the space of all smooth functions $a:\rz^n\to E$ such that 
 $$\sup_{y\in\rz^n}\|\partial^\alpha_y a(y)\| \spk{y}^{-\nu+|\alpha|}<\infty$$
for every multi-index $\alpha\in\nz_0^n$ and any choice of continuous semi-norm $\|\cdot\|$ of $E$. 
Moreover, $S^\nu_\cl(\rz^n,E)$ is the subspace consisting of all such symbols $a$ that additionally 
have asymptotic expansions into homogeneous components, in the sense that there exist 
functions $a_{(\nu-j)}\in S^{(\nu-j)}(\rz^n,E)$, $j=0,1,2,\ldots$  such that 
 $$a(y)-\sum_{j=0}^{N-1}\chi(y)a_{(\nu-j)}(y)\;\in\; S^{\nu-N}(\rz^n,E)$$
for any choice of $N\in\nz_0$ and some 0-excision function $\chi$ (i.e., $\chi$ is a smooth function 
on $\rz^n$  which vanishes near zero and is $\equiv1$ outside the unit-ball).

In case $E=\cz$ we obtain the standard classes of scalar-valued $($classical$)$ symbols. In this case we 
suppress $E=\cz$ from the notation.

\subsection{$SG$-pseudodifferential Operators}\label{sec:zero5.1}

Using the above notion of  Fr\'echet space valued symbols let us define 
 $$\smmu=\smmu(\rz^n\times\rz^n):=S^m\big(\rz^n_x,S^\mu(\rz^n_\xi)\big),\qquad 
     \smmu_\cl=\smmu_\cl(\rz^n\times\rz^n):=S^m_\cl\big(\rz^n_x,S^\mu_\cl(\rz^n_\xi)\big),$$ 
where $m$ and $\mu$ are arbitrary real numbers. We consider elements from these spaces as 
functions on $\rz^n_x\times\rz^n_\xi\to\cz$ $($the definitions can be modified in an obvious 
way to allow also values in the $(k\times k)$-matrices, covering the case of systems$)$. With 
$a\in\smmu$ we can associate sequences of symbols being homogeneous with respect to $x$ or 
$\xi$, which we shall denote by 
 $$a_{(m-j)}\in S^{\mu,(m-j)}_\cl:=S^{(m-j)}(\rz^n_x,S^\mu_\cl(\rz^n_\xi)),\qquad 
    a^{(\mu-k)}\in S^{(\mu-k),m}_\cl:=S^{(\mu-k)}(\rz^n_{\xi},S^m_\cl(\rz^n_x)),$$ 
respectively. Note that these homogeneous components are compatible in the sense that 
  $$(a^{(\mu-k)})_{(m-j)}=(a_{(m-j)})^{(\mu-k)}=: a^{(\mu-k)}_{(m-j)}\;\in\; 
      S^{(\mu-k),(m-j)}:=S^{(\mu-k)}(\rz^n_{\xi},S^{(m-j)}(\rz^n_x)).$$
We also shall make use of symbols which are smoothing in one or both variables, namely 
 $$S^{\mu,-\infty}:=\mathop{\mbox{\large$\cap$}}_{m\in\rz}\,\smmu,\qquad
   S^{-\infty,m}:=\mathop{\mbox{\large$\cap$}}_{\mu\in\rz}\,\smmu,\qquad   
   S^{-\infty,-\infty}:=\mathop{\mbox{\large$\cap$}}_{\mu,m\in\rz}\,\smmu,$$
and similarly for the subspaces of classical symbols.  

We recall the following useful characterization of classical symbols:

\begin{proposition}\label{zero5.2.5}
Let $U$ be the open unit ball in $\rz^n$ centered at 0, and $\phi:U\to\rz^n$ the diffeomorphism 
given by $\phi(y)=y/{\sqrt{1-|y|^2}}$. 
A function $a$  on $\rz^n$ belongs to  $S^0_\cl$ if and only if $a\circ\phi$ 
extends to  a smooth function on the closure $\overline{U}$.
\end{proposition}

Using that $a\in S^m_\cl$ if and only if $[\,\cdot\,]^{-m}a\in S^0_\cl$, there is an 
obvious generalization of Proposition \ref{zero5.2.5} to the case of arbitrary order. 
Moreover, 
$a\in S^{0,0}_\cl$ if and only if $b(y,\zeta):=a(\phi(y),\phi(\zeta))$ 
extends smoothly to $\overline{U}\times \overline{U}$. 

For $a\in\smmu$, the pseudodifferential operator 
$a(x,D):\calS(\rz^n)\to\calS(\rz^n)$ has continuous extensions to maps
 $$a(x,D):H^{s,\delta}(\rz^n)\longrightarrow H^{s-\re\mu,\delta-\re m}(\rz^n)$$
for the weighted Sobolev spaces
 $$H^{r,\rho}(\rz^n)=\{u\in\calS^\prime(\rz^n)\st\spk{\cdot}^{-\rho}u\in H^r(\rz^n)\},
   \qquad r,\rho\in\rz.$$
The symbol of an operator $a(x,D)$ is uniquely determined.  The class of 
$SG$-pseudodifferential operators is closed under composition and 
formal adjoints with respect to the $L_2$-scalar product on $\rz^n$. 
The standard asymptotic expansion formulae for the symbols of composition 
and adjoint are valid, e.g., given $a_j\in S^{\mu_j,m_j}$, $j=0,1$, 
the symbol $a_0\# a_1$ of the composition satisfies  
 $$a_0\#a_1(x,\xi)-\sum_{|\alpha|=0}^{N-1}
   \frac{1}{\alpha!}\partial^\alpha_\xi a_0(x,\xi) D^\alpha_x a_1(x,\xi)
   \;\in\; S^{\mu_0+\mu_1-N,m_0+m_1-N}, \quad N\in \nz_0.$$

\subsection{Parameter-ellipticity and Resolvent}\label{sec:zero5.2}

Let $\Lambda$ be a sector as in \ref{zeroi.C}. A symbol $a\in\smmu$ is called 
\textit{$\Lambda$-elliptic} if there exist constants $C,R\ge0$ such that
 $$\text{\rm spec}(a(x,\xi))\cap\Lambda=\emptyset,
   \quad\;|(x,\xi)|\ge R$$
and
 $$|(\lambda-a(x,\xi))^{-1}|\le C\spk{x}^{-m}\spk{\xi}^{-\mu},
   \qquad\lambda\in\Lambda, \ |(x,\xi)|\ge R.$$

One can always modify a $\Lambda$-elliptic symbol in order to have $R=0$ above. 
In fact, if 
\begin{equation}\label{zero5.A}
 \wt{a}(x,\xi):=a(x,\xi)+(L+1)(1-\chi(x,\xi)),\qquad L=\max_{|(x,\xi)|\leq R}|a(x,\xi)|,
\end{equation} 
with a zero excision function $\chi$ such that $\chi(x,\xi)=0$ for $|(x,\xi)|\leq R$, 
then $\wt{a}$ differs from $a$ by a remainder in
$S^{-\infty,-\infty}$ and satisfies the $\Lambda$-ellipticity conditions with $R=0$. 

We summarize the main consequences of $\Lambda$-ellipticity in the following theorem. 

\begin{theorem}\label{zero5.3}
Let $a\in S^{\mu,m}$ with $\mu>0$, $m\ge0$, be $\Lambda$-elliptic. 
\begin{itemize}
 \item[a)] $\lambda-a(x,D):H^{\mu,m}(\rz^n)\to L_2(\rz^n)$ is bijective for all 
  sufficiently large $\lambda\in\Lambda$. 
\end{itemize}
The resolvent can be described as a parameter-dependent pseudodifferential operator: 
Let $b_0(x,\xi,\lambda)=(\lambda-\wt{a}(x,\xi))^{-1}$ with $\wt{a}$ from \eqref{zero5.A}, 
and 
 $$b_{-k}(x,\xi,\lambda)\!=\!\!\sum_{\substack{j+|\alpha|=k\\j<k}}\!\frac{1}{\alpha!}
   (\partial^\alpha_\xi b_{-j})(x,\xi,\lambda)\,(D^\alpha_x\wt{a})(x,\xi)\,
    b_0(x,\xi,\lambda),\qquad k\ge1.$$
Then there exists a symbol $b(x,\xi,\lambda)$ such that 
\begin{itemize}
 \item[b)] for any $N\in\nz_0$ and all $\alpha,\beta\in\nz_0^q$ 
   $$\Big|\partial^\alpha_\xi\partial^\beta_x\Big(b(x,\xi,\lambda)-
     \sum_{k=0}^{N-1}b_{-k}(x,\xi,\lambda)\Big)\Big|
     \le 
     C\,(|\lambda|+\spk{x}^m\spk{\xi}^\mu)^{-1}\spk{x}^{-N-|\beta|}
     \spk{\xi}^{-N-|\alpha|},$$
  uniformly for $x,\xi\in\rz^n$ and $\lambda\in\Lambda$, 
 \item[c)]  $|\lambda|^2\big\{(\lambda-a(x,D))^{-1}-b(x,D,\lambda)\big\}\in 
  S^{-\infty,-\infty}$ 
  uniformly for large $\lambda\in\Lambda$. 
\end{itemize}
\end{theorem}

For a classical symbol $a$  $\Lambda$-ellipticity can be characterized by spectral properties 
of the associated homogeneous principal symbols: 

\begin{remark}\rm\label{rem:ell-classical}
Let $a\in S^{\mu,m}_\cl$. Then $a$ is $\Lambda$-elliptic if, and only if, 
\begin{itemize}
 \item[$(1)$] $\text{\rm spec}(a^{(\mu)}(x,\omega))\cap\Lambda=\emptyset$ for all
  $x\in\rz^n$ and $|\omega|=1$,
 \item[$(2)$] $\text{\rm spec}(a_{(m)}(\theta,\xi))\cap\Lambda=\emptyset$ for all
  $\xi\in\rz^n$ and $|\theta|=1$, and 
 \item[$(3)$] $\text{\rm spec}(a^{(\mu)}_{(m)}(\theta,\omega))\cap\Lambda=\emptyset$ 
  for all $|\theta|=1$ and $|\omega|=1$.
\end{itemize}
In this case the parametrix $b(x,\xi,\lambda)$ of Theorem $\mathrm{\ref{zero5.3}}$ has 
a certain classical structure; for details see \cite{MSS}. It follows that a classical symbol 
which is $\Lambda$-elliptic is also elliptic in a slightly larger sector with interior containing 
$\Lambda\setminus\{0\}$. 
\end{remark} 

\subsection{Complex Powers}\label{sec:zero5.3}

Assume from now on that $a\in S^{\mu,m}_\cl$ with $\mu>0$ and $m\ge0$ is $\Lambda$-elliptic,  
that $\lambda-a(x,D)$ is invertible for all $0\not=\lambda\in\Lambda$, and that
$\lambda=0$ is at most an isolated spectral point. 

For $z\in\cz$ with $\re z<0$ we define the complex power 
\begin{equation*}
 a(x,D)^{z}=\frac{1}{2\pi i}\int_{\partial\Lambda_\eps}
 \lambda^{z}(\lambda-a(x,D))^{-1}\,d\lambda,
\end{equation*}
where $\Lambda_\eps=\Lambda\cup\{|\lambda|\le\eps\}$ for a sufficiently small $\eps>0$, 
and $\partial\Lambda_\eps$ is a parametrization of the boundary of $\Lambda_\eps$, 
the circular part being traversed clockwise. The power $\lambda^z=e^{z\log\lambda}$ 
is determined by  the main branch of the logarithm. 
For arbitrary $z\in\cz$ we set 
 $$a(x,D)^z=a(x,D)^k\,a(x,D)^{z-k},\qquad\text{where $k\in\nz_0$ with $\re z-k<0$}.$$

Essentially by replacing the resolvent $(\lambda-a(x,D))^{-1}$ in the above Dunford 
integral by the parametrix $b(x,D,\lambda)$, one obtains: 

\begin{theorem}\label{zero5.4}
$a(x,D)^z$ is a pseudodifferential operator. If $a(x,\xi,z)$ denotes the symbol of $a(x,D)^z$ then  
$a(x,\xi,z)\in S^{\mu z,mz}_\cl$ is a holomorphic family. 
\end{theorem}

\vspace{2mm}
{\em Note}.  Most of the results in this article were obtained in 2006 as
preliminary work on a different project. Meanwhile, they have proven to be 
useful in several instances \cite{BaMZ,BaCo1,Lopes}, and it seemed appropriate
to make them available to a wider audience.

\begin{small}
\bibliographystyle{mn}

\end{small}

\end{document}